\documentclass[10pt]{amsart}
\usepackage{amsmath} 
\usepackage{amssymb}
\usepackage{bbm}
\usepackage{enumerate}
\usepackage{color}

\usepackage{graphicx}
\usepackage{xcolor}

\usepackage{esvect}

\theoremstyle{plain}
\newtheorem{theorem}{Theorem}[section]

\newtheorem{corollary}[theorem]{Corollary}
\newtheorem{lemma}[theorem]{Lemma}

\newtheorem{proposition}[theorem]{Proposition}

\theoremstyle{remark}

\newtheorem{claim}{Claim}
\newtheorem*{claim*}{Claim}

\begin{document}

\title[Polish modules over subrings of $\mathbb Q$]{Polish modules over subrings of $\mathbb Q$}

\author{Dexuan Hu} 
\author{S{\l}awomir Solecki}
\address{Department of Mathematics, 
Cornell University, Ithaca, NY 14853}
\email{dh623@cornell.edu}
\email{ssolecki@cornell.edu}

\thanks{The authors were partially supported by NSF grant DMS-1954069.}




\begin{abstract}
We give a method of producing a Polish module over an arbitrary subring of $\mathbb Q$ from an ideal of subsets of $\mathbb N$ 
and a sequence in $\mathbb N$. 
The method allows us to construct two Polish $\mathbb Q$-vector spaces, $U$ and $V$, such that 
\begin{enumerate}
\item[---] both $U$ and $V$ embed into $\mathbb R$ 
but

\item[---] $U$ does not embed into $V$ and $V$ does not embed into $U$,
\end{enumerate}
where by an embedding we understand a continuous $\mathbb Q$-linear injection. 
This construction answers a question of Frisch and Shinko \cite{FS}. In fact, our method produces a large number of incomparable with respect to 
embeddings Polish $\mathbb Q$-vector spaces. 
\end{abstract}

\maketitle

\section{Introduction}\label{S:1}

\subsection{The background for and outline of the main results} 
This is a paper on a connection between Descriptive Set Theory and Commutative Algebra. 
In \cite{FS}, Frisch and Shinko initiated the study of Polish modules and embeddings among them. 
The present work answers a question from that paper \cite[Problem~4.4]{FS} by producing incomparable with respect to embeddings Polish $R$-modules for subrings $R$ of $\mathbb Q$.
In fact, we give a rather general method of constructing Polish $R$-modules for such $R$. 
This construction associates Polish $R$-modules with certain ideals of subsets of $\mathbb N$ (translation invariant analytic P-ideals) and certain 
sequences in $\mathbb N$ (base sequences) under the assumption of a suitable coherence between the ideals and the sequences. The procedure 
appears to be quite canonical and may have other applications.

Before explaining the context of the question from \cite{FS} and our construction, which will require some definitions, 
we formulate a statement that is a special case of Corollary~\ref{T:A}. 
The statement does not require any extra definitions as it 
can be phrased entirely in the language of vector spaces over $\mathbb Q$ and the Borel structure of $\mathbb R$; on the other hand, it 
is a representative consequence of our main results.  
We treat the real numbers $\mathbb R$ as a $\mathbb Q$-vector space.

\smallskip 

\noindent {\em There exists a family $F_x$, for $x\subseteq {\mathbb N}$, of uncountable Borel $\mathbb Q$-vector subspaces of $\mathbb R$ 
such that, for all $x,y\subseteq {\mathbb N}$, 
\begin{enumerate}
\item[---] if $x\setminus y$ is finite, then $F_x\subseteq F_y$ and 

\item[---] if $x\setminus y$ is infinite, then each Borel $\mathbb Q$-linear map $F_x\to F_y$ is constantly equal to zero.
\end{enumerate}
Each $F_x$, for $x\subseteq {\mathbb N}$, can be taken to be a countable union of closed subsets of $\mathbb R$.}

\smallskip

Now, we recall the notion of Polish module from \cite{FS}. Actually, 
we recall only the special case of the definition (when the ring is commutative and countable) that is relevant to stating the problem from \cite{FS}. 
Let $R$ be a commutative countable ring (with unity $1$). By a {\bf Polish $R$-module} $M$ we understand an abelian Polish group, where $+$ denotes the group operation, and a 
continuous function 
\[
R\times M\ni (r,m) \to r\cdot m\in M,
\]
with $R$ equipped with the discrete topology, such that 
\[
\begin{split}
r\cdot (m_1+m_2) = r\cdot m_1 + r\cdot m_2,& \;\; (r_1+r_2)\cdot m = r_1\cdot m+r_2\cdot m,\\ 
r_1\cdot(r_2\cdot m) = (r_1r_2)\cdot m,& \;\; 1\cdot m=m, 
\end{split}
\]
for all $r, r_1, r_2\in R$ and $m, m_1, m_2\in M$. 
Of course, if $R$ is a countable field, then an $R$-module is a vector space over $R$. 
For the record, we also recall the notion of homomorphism and embedding among $R$-modules. If $M_1$ and $M_2$ are $R$-modules, 
a function $f\colon M_1\to M_2$ is called an {\bf $R$-module homomorphism} 
if it is a group homomorphism and $f(rm)= rf(m)$ for all $r\in R$ and $m\in M_1$. An injective $R$-module homomorphism is called an {\bf $R$-embedding}. 
We use the suggestive notation from \cite{FS} 
\[
M_1\sqsubseteq^R M_2
\]
to indicate that there is a continuous $R$-embedding $M_1\to M_2$; and also  
\[
M_1\sqsubset^R M_2
\]
if $M_1\sqsubseteq^R M_2$ and $M_2\not\sqsubseteq^R M_1$. 

It was proved in \cite[Theorem~1.2 and Section~3]{FS} that 

\smallskip
\noindent {\em if $R$ is a countable Noetherian ring, then there is a family of uncountable Polish $R$-modules $M_S$, 
parametrized by nontrivial quotients $S$ of $R$, such that, for each uncountable Polish $R$-module $M$, we have $M_S\sqsubseteq^R M$ 
for some $S$.} 

\smallskip 

\noindent Recall that a commutative ring is {\bf Noetherian} if each strictly increasing under inclusion sequence of ideals is finite.
The modules $M_S$ have an explicit definition, but it will not be relevant here. 
The family of nontrivial quotients $S$ of $R$ is countable; thus, the above statement asserts that the class of uncountable Polish $R$-modules 
has a countable basis under $R$-embeddings. The situation is even more striking when $R$ is a countable field. A field $R$ has only one non-trivial quotient, namely, $R$ itself, 
and is obviously Noetherian. So the above family of the $M_S$-s has only one element $M_R$. 
Thus, 

\smallskip
\noindent {\em if $R$ is a countable field, then there is an uncountable Polish $R$-vector space $M_R$ 
such that, for each uncountable Polish $R$-vector space $M$, we have $M_R\sqsubseteq^R M$.}
\smallskip

\noindent So the class of uncountable Polish $R$-vector spaces has a one-element basis under embedding.

A particularly interesting case is that of $R={\mathbb Q}$. In this case, there is a clasical uncountable Polish $\mathbb Q$-vector space, namely, $\mathbb R$. 
The above statement gives $M_{\mathbb Q}\sqsubseteq^{\mathbb Q} {\mathbb R}$. As it was pointed out in \cite{FS}, 
there is no continuous $\mathbb Q$-embedding from $\mathbb R$ to $M_{\mathbb Q}$, so we have 
\[
M_{\mathbb Q} \sqsubset^{\mathbb Q} {\mathbb R}. 
\]
Now, a natural question arises whether there exists a Polish $\mathbb Q$-vector space $V$ such that 
\begin{equation}\label{E:que}
M_{\mathbb Q} \sqsubset^{\mathbb Q} V \sqsubset^{\mathbb Q} {\mathbb R}. 
\end{equation} 
This question was asked as \cite[Problem~4.4]{FS}. We answer it in the affirmative by proving the following theorem.

\begin{theorem}\label{T:B}
Let $R$ be a subring of $\mathbb Q$ not equal to $\mathbb Z$. 
There exists a family $V_x$, for $x\subseteq{\mathbb N}$, of uncountable Polish $R$-modules such that, for all $x,y\subseteq {\mathbb N}$, we have 
\begin{enumerate}
\item[(i)] $V_x\sqsubseteq^R {\mathbb R}$, ; 

\item[(ii)] if $x\setminus y$ is finite, then $V_x\sqsubseteq^R V_y$; 

\item[(iii)] if $x\setminus y$ is infinite, then each continuous $R$-module homomorphism $V_x\to V_y$ is identically equal to zero; in particular, $V_x\not\sqsubseteq^R V_y$. 
\end{enumerate}

In fact, there exist continuous $R$-module embeddings $f_x\colon V_x\to {\mathbb R}$ witnessing (i) so that (ii) is witnessed by a continuous $R$-module emdeding
$g\colon V_x\to V_y$ with $f_x=f_y\circ g$. Further, the image of $f_x$ is a countable union of closed subsets of $\mathbb R$. 
\end{theorem} 

To see clearly the relationship between Theorem~\ref{T:B} and \cite[Problem~4.4]{FS}, take $R={\mathbb Q}$ and 
note that if $x\subseteq {\mathbb N}$ is such that both $x$ and ${\mathbb N}\setminus x$ are infinite, then $V_x$ as in Theorem~\ref{T:B} fulfills \eqref{E:que}. Indeed, by 
Theorem~\ref{T:B}(i) and the defining property of $M_{\mathbb Q}$, we have 
\[
M_{\mathbb Q} \sqsubseteq^{\mathbb Q} V_x\sqsubseteq^{\mathbb Q} {\mathbb R}. 
\]
Set $y= {\mathbb N}\setminus x$. If ${\mathbb R}\sqsubseteq^{\mathbb Q} V_x$, then, by Theorem~\ref{T:B}(i) for $y$ and transitivity of $\sqsubseteq^{\mathbb Q}$, 
we would have
$V_y\sqsubseteq^{\mathbb Q} V_x$ contradicting Theorem~\ref{T:B}(iii). If $V_x\sqsubseteq^{\mathbb Q} M_{\mathbb Q}$, then, by definition of $M_{\mathbb Q}$ and 
transitivity of $\sqsubseteq^{\mathbb Q}$, we would have $V_x\sqsubseteq^{\mathbb Q}V_y$ again contradicting Theorem~\ref{T:B}(iii). It follows that \eqref{E:que} 
holds with $V= V_x$. 

As pointed out in Section~\ref{Su:bas}, subrings of $\mathbb Q$ are Noetherian; therefore, the above theorem fits into the context from \cite{FS} described above 
even for $R$ not equal to $\mathbb Q$. 

We have a corollary of Theorem~\ref{T:B} that is phrased in terms of the Borel structure of $\mathbb R$ only, and which generalizes the statement at the top of the introduction.

\begin{corollary}\label{T:A} 
Let $R$ be a subring of $\mathbb Q$ not equal to $\mathbb Z$. 
There exists a family $F_x$, for $x\subseteq {\mathbb N}$, of uncountable Borel $R$-submodules of $\mathbb R$ 
such that, for all $x,y\subseteq {\mathbb N}$, 
\begin{enumerate}
\item[(i)] if $x\setminus y$ is finite, then $F_x\subseteq F_y$ and 

\item[(ii)] if $x\setminus y$ is infinite, then each Borel $R$-module homomorphism $F_x\to F_y$ is constantly equal to zero. 
\end{enumerate}
Each $F_x$, for $x\subseteq {\mathbb N}$, can be taken to be a countable union of closed sets. 
\end{corollary}

Note that that the Borel complexity (countable unions of closed sets) in the conclusion of Corollary~\ref{T:A} is clearly best possible. 

\begin{proof}[Proof of Corollary~\ref{T:A} from Theorem~\ref{T:B}] Let $V_x$ and $f_x$ for $x\subseteq {\mathbb N}$, be as in Theorem~\ref{T:B}. 
Let 
\[
F_x = f_x ( V_x).
\]
Clearly $F_x$ is an $R$-submodule of $\mathbb R$. Directly from the properties of $f_x$ in Theorem~\ref{T:B}, $F_x$ is $F_\sigma$ and (i) holds. 
If $x\setminus y$ is infinite and 
$f\colon F_x\to F_y$ is a Borel $R$-module homomorphism, then $(f_y)^{-1}\circ f \circ f_x$ is an $R$-module homomorphism from 
$V_x$ to $V_y$, which is Borel since the inverse of a Borel bijection is Borel. Thus, by Pettis Theorem \cite[Theorem 9.10]{Ke}, 
it is continuous, and so $f$ is constantly equal to $0$ by Theorem~\ref{T:B}. 
\end{proof}

The paper is structured as follows. 
Theorem~\ref{T:B} is a consequence of a general method of constructing Polish modules over subrings of $\mathbb Q$, which we present 
in Section~\ref{S:2}; see especially Theorem~\ref{T:C}. Then, in Section~\ref{S:3}, we prove a result, Theorem~\ref{T:D}, 
on strong non-containment, under appropriate conditions, of modules produced 
using the method from Section~\ref{S:2}. Finally, in Section~\ref{S:4}, we apply the results from Sections~\ref{S:2} and \ref{S:3} 
to prove Theorem~\ref{T:B}.

\subsection{Conventions} 
By $\mathbb N$ we denote the set of all positive integers; so $0\not\in {\mathbb N}$. 
The power sets of $\mathbb N$, ${\mathcal P}({\mathbb N})$, is often regarded as a compact metric space 
via the canonical identification with $\{ 0,1\}^{\mathbb N}= 2^{\mathbb N}$. 
Further, ${\mathbb R}_{\geq 0}$ stands for the set 
of non-negative real numbers. 
To keep notation lighter, for $r,s\in {\mathbb R}$,  set 
\begin{equation}\label{E:abm}
r\ominus s = |r-s|. 
\end{equation} 
Finally, a subset of a metric space is called 
\[
F_\sigma
\] 
if it is a countable union of closed sets.

\section{A class of Polish modules}\label{S:2}

In this section, we present the main construction of the paper. It associates a Polish module with two inputs: 
an ideal (of a certain type) of subsets of $\mathbb N$ and a base sequence adapted to the ideal. 
These two main objects in the construction---ideals of subsets of $\mathbb N$ and base sequences---are defined in Section~\ref{Su:baid}. 
After some preliminary work in Sections~\ref{Su:bas} and \ref{Su:id}, we define the Polish modules and prove their basic properties in Section~\ref{Su:pom}. 

Our construction extends the one in \cite[Section~3]{So2}, where Polishable subgroups of $\mathbb R$ were associated with certain ideals of subsets of $\mathbb N$.

\subsection{Base sequences and ideals of sets of natural numbers}\label{Su:baid}

A sequence 
\[
\vv{a}= (a_n)
\] 
is called a {\bf base sequence} if $a_n\in {\mathbb N}$ and $a_n\geq 2$, for each $n\in {\mathbb N}$. 
A non-empty family $I$ of subsets of $\mathbb N$ is called an {\bf ideal} if it is closed under taking finite unions and subsets. 
Additionally, in this paper, we will always assume that $\{ n\}\in I$ for each $n\in {\mathbb N}$. 
The two types of objects are related to each other through the following notion. A base sequence $\vv a$ is called {\bf adapted to} $I$ if 
\[
\{ n\mid a_n\not= a_{n+1}\} \in I. 
\]

\subsection{Base sequences}\label{Su:bas}

Base sequences are useful to us for two reasons---each such sequence gives a representation of real numbers and it determines a subring of $\mathbb Q$.

\subsubsection{Representations of non-negative reals}

Given a base sequence $\vv a$, each $r\in {\mathbb R}_{\geq 0}$ has a unique representation as 
\begin{equation}\label{E:real}
r = [r] + \sum_{n=1}^\infty \frac{r_n}{a_1\cdots a_n}, 
\end{equation} 
where $[r]$ is the integer part of $r$ and $r_n$ is such that $0\leq r_n <a_n$ and $r_n\not= a_n-1$ for infinitely many $n$. 
We leave checking the above assertion to the reader. Obviously, this is analogous to representing $r$ using its decimal expansion. One can extend the representation \eqref{E:real} 
to negative reals, but we will not need this extention in what follows. For $r\in {\mathbb R}_{\geq 0}$ and $n\in {\mathbb N}$, $r_n$ will stand for the $n$-th digit of $r$ in 
the representation of $r$ as in \eqref{E:real}.

Fix a base sequence $\vv a$. 
\begin{lemma}\label{L:rul}  
Let $r, s\in{\mathbb R}_{\geq 0}$. 
\begin{enumerate}
\item[(i)] If $0 \leq r < \frac{1}{a_1a_2\cdots a_l}$, then $r_i=0$ for all $i \leq l$. 
\item[(ii)] If $0 \leq s-r < \frac{1}{a_1a_2\cdots a_l}$ and $(r_l\not= a_l-1$ or $s_l\not= 0)$, 
then  $s_i= r_i$ for all $i < l$.
\end{enumerate}
\end{lemma} 

\begin{proof} (i) is immediate. To see point (ii), note that, by (i), $(s-r)_i=0$ for all $i\leq l$. By a straightforward calculation, 
this condition, in conjunction with  $r_l\not= a_l-1$ or $s_l\not=0 $, implies that $s_i=r_i$ for all $i<l$. 
\end{proof} 

The next lemma on the continuity of the digits of a number is essentially a qualitative version of Lemma~\ref{L:rul}(ii). 

\begin{lemma} \label{L:conv digi}
 Let $\big(r(k)\big)_{k \in \mathbb N}$, be a sequence in $\mathbb R_{\geq 0}$ converging to $r\in {\mathbb R}_{\geq 0}$. 
 If either one of the following two conditions holds 
 \begin{enumerate}
     \item[(i)] $r$ has infinitely many non-zero digits,

     \item[(ii)] $r(k) \geq r$ for all $k$,
 \end{enumerate}
 then, for each $n \in \mathbb N$, ${r(k)}_n = r_n \hbox{ for large enough } k$.
\end{lemma}

\begin{proof} 
Under assumption (ii), for each 
$l$ with $r_l\not= a_l-1$, Lemma~\ref{L:rul}(ii) implies that, for large enough $k$, we have $r(k)_i=r_i$ for all $i<l$. Since there are infinitely many such $l$, the conclusion follows. 
Now, assume (i). By what we just proved, we can suppose that $r(k)<r$ for all $k$. Now, for each 
$l$ with $r_l\not=0$, Lemma~\ref{L:rul}(ii) implies that, for large enough $k$, we have $r(k)_i=r_i$ for all $i<l$. Since, by assumption (i), 
there are infinitely many such $l$, the conclusion follows. 
\end{proof}

For $r\in {\mathbb R_{\geq 0}}$ with the representation \eqref{E:real}, we let 
\begin{equation}\label{E:jdf}
j_{\vv a}(r) = \{ n\in {\mathbb N}\mid r_n\not= r_{n+1}\}, 
 \end{equation}
so $j_{\vv a}\colon {\mathbb R}_{\geq 0} \to 2^{\mathbb N}$. 
Lemma~\ref{L:conv digi} has an immediate corollary concering continuity properties of $j_{\vv a}$.

\begin{lemma} \label{L:contj}
 Let $\big( r(k)\big)_{k \in \mathbb N}$ be a sequence in $\mathbb R_{\geq 0}$ converging to $r\in {\mathbb R}_{\geq 0}$. 
 If either one of the following two conditions holds 
 \begin{enumerate}
     \item[(i)] $r$ has infinitely many non-zero digits,

     \item[(ii)] $r(k) \geq r$ for all $k$,
 \end{enumerate}
 then $j_{\vv a}(r(k)) \to j_{\vv a}(r) \hbox{ as } k \to \infty$.
\end{lemma}

The next lemma describes how the function $j_{\vv a}$ interacts with the algebraic operations of addition and subtraction. Recall $\ominus$ from \eqref{E:abm}.

\begin{lemma}\label{L:jalg} 
For $r,s\in {\mathbb R}_{\geq 0}$, $j_{\vv a}(r+s)$ and $j_{\vv a}(r \ominus s)$ are subsets of the union of the following three sets 
    \begin{equation}\label{E:thrs} 
    \begin{split} 
    &j_{\vv a}(r) \cup j_{\vv a}(s),\\  
    &\{n \mid n+1 \in j_{\vv a}(r) \cup j_{\vv a}(s)\},\\
    &\{n \mid a_n \neq a_{n+1}\} \cup \{n \mid a_{n+1} \neq a_{n+2}\}.
    \end{split}
    \end{equation} 
\end{lemma} 

\begin{proof} 
Suppose $n$ is not in any of the sets in \eqref{E:thrs}, that is, 
$r_n = r_{n+1} = r_{n+2}$, $s_n = s_{n+1} = s_{n+2}$ and $a_n = a_{n+1} = a_{n+2}$. By a direct computation of the digits of $r+s$ and $r \ominus s$ using the digits of $r$ and $s$, 
we see that $(r + s)_n = (r + s)_{n+1}$ and $(r \ominus s)_n = (r \ominus s)_{n+1}$, so $n \notin j_{\vv a}(r + s)$ and $n \notin j_{\vv a}(r \ominus s)$. 
\end{proof}

\subsubsection{Subrings of $\mathbb Q$}

Given a base sequence $\vv a$, we define a subring ${\mathbb Q}_{\vv a}$ of $\mathbb Q$. 
First, however, we make general comments on subrings of $\mathbb Q$. For an arbitrary set $P$ of primes, define 
\[
P^{-1}{\mathbb Z}= \{ \frac{k}{l}\mid k\in {\mathbb Z},\, l\in {\mathbb N}, \hbox{ and, } \hbox{for each prime }p, \hbox{ if } p\mid l, \hbox{ then } p\in P\}. 
\]
Note that $P^{-1} {\mathbb Z}$ is a subring of ${\mathbb Q}$. The following lemma is surely well known. We give an argument justifying it for the convenience of the reader. 

\begin{lemma}\label{L:subr} 
Each subring of $\mathbb Q$ is of the form $P^{-1}{\mathbb Z}$, for a suitable set $P$ of primes. 
\end{lemma}

\begin{proof} 
Let $R$ be a subring of $\mathbb Q$. It suffices to show that if $m\in {\mathbb Z}$ and $n\in {\mathbb N}$ are relatively prime, 
$m/n\in R$, and a prime $p$ divides $n$, then 
$1/p\in R$. To see this, note that $n/p$ is in $\mathbb N$, and, 
by adding $m/n$ to itself $n/p$ many times, we get $m/p\in R$. Since $p\nmid m$, there are $a,b\in {\mathbb Z}$ such that $1= ap+bm$. So we have 
\[
\frac{1}{p} = \frac{ap+bm}{p} = a + b \,\frac{m}{p}, 
\]
from which we get $1/p\in R$, as required. 
\end{proof} 

Note that if we take $P=\emptyset$ in the lemma above, then $P^{-1} {\mathbb Z} = {\mathbb Z}$; if $P$ is the set of all primes, then 
$P^{-1}{\mathbb Z}= {\mathbb Q}$. Furthermore, each subring of $\mathbb Q$ is Noetherian as each ideal in $P^{-1} {\mathbb Z}$ is of the form  
\[
n\cdot \big( P^{-1} {\mathbb Z}\big), 
\]
where $n=0$ or $n$ is a positive integer such that $p\nmid n$ for each $p\in P$; see \cite[Proposition~3.11]{AM}. (We thank Patrick Allen for this reference.) 

For a base sequence $\vv a$, let 
\[
{\rm pr}({\vv{a}}) = \{ p \mid \; p \hbox{ a prime and } p\mid a_n \hbox{ for all but finitely many }n\}.
\]
Define 
\[
{\mathbb Q}_{\vv a} = \big( {\rm pr}({\vv a})^{-1}\big) {\mathbb Z}.
\]
We have the following immediate consequence of Lemma~\ref{L:subr}. 

\begin{proposition}
Each subring of $\mathbb Q$ is of the form ${\mathbb Q}_{\vv a}$ for some base sequence $\vv a$. 
\end{proposition}

\subsection{Ideals of subsets of $\mathbb N$}\label{Su:id} 

Two properties of ideals of subsets of $\mathbb N$ will be important in the sequel: translation invariance and being an analytic P-ideal.

An ideal $I$ is call {\bf translation invariant} if, for each $x\in I$, we have 
\[
\{ n+1\mid n\in x\}\in I\;\hbox{ and }\; \{ n\mid n+1\in x\} \in I.
\]

Before we talk about analytic P-ideals, we need to introduce appropriate submeasures. 
By a {\bf lower semicontinuous} ({\bf lsc}) {\bf submeasure} we understand a lower semicontinuous function $\phi\colon 2^{\mathbb N}\to [0, \infty]$ such that 
$\phi(x)\leq \phi(y)$ if $x\subseteq y$ and $\phi(x\cup y)\leq \phi(x)+\phi(y)$ for all $x,y$. In this paper, we always assume that $0<\phi(\{ n\})<\infty$ for each $n\in {\mathbb N}$. 
With each lsc submeasure we associate an ideal 
\[
{\rm Exh}(\phi) = \{ x\subseteq {\mathbb N}\mid \phi(x\setminus \{1, \dots, n\})\to 0, \hbox{ as }n\to \infty\}. 
\]

An ideal $I$ is an {\bf analytic P-ideal} 
if it is analytic as a subsets of $2^{\mathbb N}$ and has the following property: 
for each sequence $(x_n)$ of elements of $I$, there exists $y\in I$ such that $x_n\setminus y$ is finite for each $n$. 
By \cite[Theorem~3.1]{So1}, an ideal $I$ is an analytic P-ideal if and only if it is of the form 
\[
I = {\rm Exh}(\phi), 
\]
for some lsc submeasure $\phi$. 

We have the following lemma on basic properties of ideals of the form 
${\rm Exh}(\phi)$. All of these properties have been noticed in the literature. We recall them here for the reader's convenience.

\begin{lemma}\label{L:phs} 
Let $\phi$ be a lsc submeasure. Set $I= {\rm Exh}(\phi)$. 
\begin{enumerate}
\item[(i)] If $\psi$ is a lsc submeasure with $I= {\rm Exh}(\psi)$, then, for each $\epsilon>0$, there exists $\delta>0$ such that if 
$x\subseteq {\mathbb N}$ and $\phi(x)<\delta$, then $\psi(x)<\epsilon$. 

\item[(ii)] For each $\epsilon>0$, there exists $\delta>0$ such that if 
$x, y\subseteq {\mathbb N}$ and $\phi(x), \phi(y)<\delta$, then $\phi(x\cup y)<\epsilon$. 

\item[(iii)] If $I$ is translation invariant, then, for each $\epsilon>0$, there exists $\delta>0$ such that if 
$x\subseteq {\mathbb N}$ and $\phi(x)<\delta$, then 
\[
\phi\big(\{ n\in {\mathbb N}\mid n+1\in x\}\big)<\epsilon\;\hbox{ and }\; \phi\big(\{ n+1\in {\mathbb N}\mid n\in x\}\big)<\epsilon. 
\]
\end{enumerate} 
\end{lemma}

\begin{proof} 
All three points are consequences of the following statements. For $\emptyset \not= x_n\subseteq {\mathbb N}$, $n\in {\mathbb N}$, 
\begin{enumerate} 
\item[---] if $\phi(x_n)\to 0$ as $n\to \infty$, then $\min_n x_n\to \infty$ as $n\to \infty$; 

\item[---] if $\sum_n \phi(x_n)<\infty$, then $\bigcup_n x_n\in I$.
\end{enumerate} 
The first one of these statements follows from our assumption that $\phi(\{ k\})>0$ for each $k\in {\mathbb N}$, while the second one 
is a consequence of semicontinuity of $\phi$. 
\end{proof}

\subsection{The module associated with a base sequence and an ideal}\label{Su:pom}

Let $\vv a$ be a base sequence, which we fix for the duration of this section. Recall the definition \eqref{E:jdf} of $j_{\vv a}$, and set 
\[
j= j_{\vv a}.
\]
We also fix an ideal $I$, to which the base sequence $\vv a$ is adapted.
Define the following subset of $\mathbb R$ 
\begin{equation}\label{E:hh}
H= \{ r\mid  r\in {\mathbb R}_{\geq 0} \hbox{ and } j(r)\in I\}\cup \{ -r\mid r \in {\mathbb R}_{\geq 0} \hbox{ and } j(r)\in I\}. 
\end{equation} 
The set above will be the underlying subset of the Polish module we will define below. 

We can immediately point out a complexity estimate on the set $H$ that will be useful later on. 

\begin{lemma}\label{L:hcomp} 
If $I$ is an $F_\sigma$ ideal, then $H$ is an $F_\sigma$ subset of $\mathbb R$. 
\end{lemma}

\begin{proof} Set 
\[
{\mathbb R}_\infty = \{ r\in {\mathbb R}_{\geq 0}\mid r \hbox{ has infinitely many non-zero digits}\}. 
\]

Obviously, it will suffice to show that the set 
\[
j^{-1}(I) = \{ r\in {\mathbb R}_{\geq 0} \mid j(r)\in I\}
\]
is $F_\sigma$. Since $I$ is an $F_\sigma$ ideal, we have that 
\begin{equation}\notag
I= \bigcup_n F_n
\end{equation} 
with $F_n\subseteq 2^{\mathbb N}$ closed. Let $H_n$ be the closure in ${\mathbb R}_{\geq 0}$ of $j^{-1}(F_n)$. 
Directly from Lemma~\ref{L:contj}(i), we get 
\[
H_n\cap {\mathbb R}_\infty \subseteq j^{-1}(F_n). 
\]
It follows from this inclusion that 
\[
j^{-1}(I)\cap R_{\infty} =  \big( \bigcup_n j^{-1}(F_n)\big) \cap {\mathbb R}_\infty \subseteq (\bigcup_n H_n)\cap {\mathbb R}_\infty \subseteq \bigcup_n j^{-1}(F_n) = j^{-1}(I). 
\]
Thus, since ${\mathbb R}_{\geq 0}\setminus {\mathbb R}_\infty \subseteq j^{-1}(I)$ (as $j(r)$ is finite for each $r\in {\mathbb R}_{\geq 0}\setminus {\mathbb R}_\infty$), we have 
\[
j^{-1}(I)= \big( \bigcup_n H_n\big) \cup \big( {\mathbb R}_{\geq 0}\setminus {\mathbb R}_\infty\big). 
\]
Since ${\mathbb R}_{\geq 0}\setminus {\mathbb R}_\infty$ is countable and each $H_n$ is closed, we see that $j^{-1}(I)$ is $F_\sigma$. 
\end{proof}

\subsubsection{Algebraic operations on $H$} 

In this section, we assume that the ideal $I$ is translation invariant.

\begin{lemma}\label{L:base1}
\begin{enumerate} 
\item[(i)] $H$ is a subgroup of $\mathbb R$ taken with addition $+$. 

\item[(ii)] $H$ is closed under the multiplication by elements of ${\mathbb Q}_{\vv a}$.
\end{enumerate}
\end{lemma} 

\begin{proof} Let 
    \[
    H_{\geq 0}= \{ r\in {\mathbb R}\mid r\geq0\hbox{ and } j(r)\in I\}. 
    \]

    (i) Recall the operation $\ominus$ from \eqref{E:abm}. It suffices to show that $H_{\geq 0}$  is closed under $+$ and $\ominus$, which follows from Lemma~\ref{L:jalg} 
    since $\vv a$ is adapted to $I$ and the ideal $I$ is translation invariant.
    
    (ii) It is sufficient to show that for each $p \in {\rm pr}({\vv{a}})$ and  $r\in H_{\geq 0}$
\begin{equation}\label{E:hpo}    
\frac{1}{p}r \in H_{\geq 0}.
\end{equation} 
If $\frac{1}{p}r$ has finitely many non-zero digits, then $j(\frac{1}{p}r)$ is finite, so $\frac{1}{p}r \in H_{\geq 0}$. Thus, in the remainder of this proof, we assume that 
$\frac{1}{p}r$ has infinitely many non-zero digits. 

Let $N \in \mathbb N$ such that $p \mid a_n$ for all $n \geq N$. The existence of such $N$ follows from $p$ belonging to ${\rm pr}({\vv{a}})$.    
Now, \eqref{E:hpo} will be a consequence of the inclusion 
\begin{equation}\label{E:hpof}
j\big(\frac{1}{p}r\big) \subseteq 
j(r) \cup \{n+1 \mid n \in j(r)\} \cup \{n \mid a_n \neq a_{n+1}\} \cup \{1,\dots,N\},
\end{equation} 
since ${\vv{a}}$ is adapted to $I$, $j(r)\in I$, and $I$ is translation invariant.

The proof of \eqref{E:hpof} is done by computing the digits of $\frac{1}{p}r$ from the digits of $r$.
We need two pieces of notation. Let    
    \begin{equation}\label{E:trucr}
     [r]_n = [r]+ \sum_{k=1}^n \frac{r_k}{a_1\cdots a_k}, 
\end{equation}
and let ${r'_n}$ be the integer such that
    \[
    0 \leq {r'_n} < p\; \hbox{ and }\; {r'_n} \equiv r_n \mod p.
    \]
Also, the following identity, that is easy to see, will be useful 
\begin{equation}\label{E:idr} 
\frac{r_n}{a_1a_2\cdots a_np} = \frac{\big[\frac{r_n}{p}\big]}{a_1a_2 \cdots a_n}+ \frac{\frac{a_{n+1}}{p} r'_n}{a_1a_2\cdots a_{n+1}}.
\end{equation}

To get the relevant digits of $\frac{1}{p}r$, we first find certain digits of $\frac{1}{p}[r]_n$, for $n\geq N+1$.
Observation \eqref{E:aaa} is verified by induction below: 
\begin{equation}\label{E:aaa} 
\begin{split}
&\hbox{for $n \geq N+1$, the $n$-th and $n+1$-st digits of $\frac{1}{p}[r]_{n}$ are}\\
&\big[\frac{r_n}{p}\big]+\frac{a_n}{p}{r'_{n-1}} \hbox{ and }  \frac{a_{n+1}}{p} r'_n, \hbox{ respectively,}\\ 
&\hbox{and all the digits following them are $0$.}
\end{split} 
\end{equation}
Before we start proving \eqref{E:aaa}, observe that a quick calculation shows 
\begin{equation}\label{E:dig} 
0\leq \big[\frac{r_{n}}{p}\big]+\frac{a_n}{p}{r'_{n-1}} < a_n. 
\end{equation} 
To obtain the base case $n=N+1$ of \eqref{E:aaa}, 
note that $a_1a_2\cdots a_{N}\frac{1}{p}[r]_{N-1}$ is an integer, which implies that for all $n \geq N+1$, the $n$-th digit of $\frac{1}{p}[r]_{N-1}$ is $0$. 
By adding $\frac{r_{N}}{a_1a_2\cdots a_{N} p}$ to $\frac{1}{p}[r]_{N-1}$ and using \eqref{E:idr} with $n=N$, 
we see that $\frac{1}{p}[r]_{N}$ has the $N+1$-st digit $\frac{a_{N+1}}{p}r'_N$ and all the digits after that are $0$. Using \eqref{E:idr} with $n = N+1$ to add 
$\frac{r_{N+1}}{a_1a_2\cdots a_{N+1} p}$ to $\frac{1}{p}[r]_{N}$ and remembering \eqref{E:dig}, 
we obtain the base case. The inductive step is straightforward from the inductive hypothesis and \eqref{E:idr}. 

From \eqref{E:aaa} and \eqref{E:idr} we see that, for all $n \geq N+1$, adding $\frac{r_{n+1}}{a_1a_2\cdots a_{n+1} p}$ to $\frac{1}{p}[r]_{n}$ 
does not change the $i$-th digit for $i \leq n$. Putting together this observation and \eqref{E:aaa}, we see that, for $n \geq N+1$, the $n$-th digit of $\frac{1}{p}[r]_{k}$ for $k\geq n$ is 
\begin{equation}\label{E:digit}
\big[\frac{r_n}{p}\big]+\frac{a_n}{p}{r'_{n-1}}.
\end{equation}   
Since $\frac{1}{p}r$ is assumed to have infinitely many non-zero digits and 
$\frac{1}{p}[r]_{k} \to \frac{1}{p}r$ as $k \to \infty$, by Lemma~\ref{L:conv digi}(i), for each $n$, the $n$-th digit of $\frac{1}{p}r$ is equal to that of 
$\frac{1}{p}[r]_{k}$ for large $k$. In particular, for $n \geq N+1$, the $n$-th digit of $\frac{1}{p}r$ is given by \eqref{E:digit}. 
It follows that for $n \geq N+1$ if $r_{n-1} = r_{n} = r_{n+1}$ and $a_n = a_{n+1}$, then the $n$-th digit of $\frac{1}{p}r$ is the same as its $n+1$-st digit; that is, 
if $n-1 \notin j(r)$ and $n \notin j(r)$ and $a_{n} = a_{n+1}$, then $n \notin j\big(\frac{1}{p}r\big)$, and \eqref{E:hpof} follows. 
\end{proof}

\subsubsection{Topology on $H$}

In this section, we assume that the ideal $I$ is translation invariant and an analytic P-ideal.

Our goal is to describe a topology on $H$. We will need the following general lemma that is essentially due to Chittenden \cite{Ch}. 
We derive it here from the basic theory of uniformities, for which we refer the reader to \cite[Section~8.1]{En}.

\begin{lemma}\label{L:metr} 
Let $X$ be a set and $\rho\colon X\times X\to {\mathbb R}$ a function such that 
\begin{enumerate}
\item[(i)] for all $x,y\in X$, $\rho(x,y)\geq 0$, and $\rho(x,y)=0$ if and only if $x=y$;

\item[(ii)] for all $x,y\in X$, $\rho(x,y)=\rho(y,x)$; 

\item[(iii)] for each $\epsilon>0$, there exists $\delta>0$ such that, for all $x,y, z\in X$, if $\rho(x,y)<\delta$ and $\rho(y,z)<\delta$, then $\rho(x,z)<\epsilon$. 
\end{enumerate} 
Then there exists a metric $d$ on $X$ such that, for each $\epsilon >0$, there exists $\delta>0$ with
\[
\big(\rho(x,y)<\delta \Rightarrow d(x,y)<\epsilon\big) \hbox{ and } \big(d(x,y)<\delta \Rightarrow \rho(x,y)<\epsilon\big), \hbox{  for all }x,y\in X.
\]
\end{lemma} 

\begin{proof} 
For $A\subseteq X\times X$, let 
\[
A^{-1} = \{ (y,x)\mid (x,y)\in A\}. 
\]

Define 
\begin{equation}\label{E:unde} 
V_\epsilon = \{ (x,y)\in X\times X\mid  \rho(x,y) <\epsilon \},\; \hbox{ for }\epsilon >0.
\end{equation} 
The family $\mathcal U$ consisting of all subsets $A$ of $X\times X$ such that $A=A^{-1}$ and 
$A\supseteq V_\epsilon$ for some $\epsilon>0$ is a uniformity on $X$. 
To show this, we need to see the following three properties: 
\begin{enumerate}
\item[(i)] $\bigcap_{\epsilon>0} V_\epsilon = \{ (x,y)\in X\times X\mid x=y\}$; 

\item[(ii)] $V_\epsilon^{-1}  = V_{\epsilon}$;

\item[(iii)] for each $\epsilon>0$, there exists $\delta>0$ such that 
\[
\{ (x,z) \mid  (x,y), (y,z)\in V_\delta, \hbox{ for some } y\in X\} \subseteq V_\epsilon.
\]
\end{enumerate}
These properties are immediate corollaries of the corresponding properties of $\rho$ from the assumptions of the lemma. By \cite[Theorem~8.1.21]{En}, 
the uniformity $\mathcal U$ is given by a metric $d$ on $X$. The conclusion of the lemma follows from the definition of the sets $V_\epsilon$. 
\end{proof}

Since $I$ is an analytic P-ideal, we have $I = {\rm Exh}(\phi)$, for some lsc submeasure $\phi$.  We fix such $\phi$. Recall the operation $\ominus$ from \eqref{E:abm}, and 
define 
\[
\rho(r,s) = \phi\big( j(r\ominus s)\big) + (r\ominus s), \; \hbox{ for } r,s \in \mathbb R.
\]
The next result gives a weak form of the triangle inequality for $\rho$ as in Lemma~\ref{L:metr}(iii). 

\begin{lemma}\label{L:rhotri}
For any $\epsilon > 0$ there exists $\delta > 0$ such that for any $r,s,t$ in $H$,
\[
\rho(r,s), \; \rho(s,t) < \delta \Rightarrow \rho(r,t) < \epsilon.
\]
\end{lemma}

\begin{proof}
Fix $\epsilon>0$. Note that $r \ominus t$ is equal to one of the following
\[
(r \ominus s) + (s \ominus t), \;  (r \ominus s) \ominus (s \ominus t).
\]
By Lemma~\ref{L:jalg}, $j(r \ominus t)$ is included in the union of the following three sets
\begin{equation}\label{E:thrs2}
    \begin{split} 
    &j(r\ominus s) \cup j(s \ominus t),\\  
    &\{n \mid n+1 \in j(r\ominus s) \cup j(s \ominus t )\},\\
    &\{n \mid a_n \neq a_{n+1}\} \cup \{n \mid a_{n+1} \neq a_{n+2}\}.
    \end{split}
\end{equation} 
By Lemma~\ref{L:phs}(ii) and (iii), there is $\delta>0$ such that 
\begin{equation}\label{E:minus1}
\phi\big(j(r\ominus s)\big), \; \phi\big(j(s\ominus t)\big) < \delta
\Rightarrow 
\phi\big(\{n \mid n+1 \in j(r\ominus s)\cup j(s \ominus t )\}\big) < \epsilon.
\end{equation}
Moreover, since $\{n \mid a_n \neq a_{n+1}\} \cup \{n \mid a_{n+1} \neq a_{n+2}\}$ is in $I$, there is $N$ large enough such that 
\begin{equation}\label{E:tail}
\phi\big(\{n \mid a_n \neq a_{n+1}\} \cup \{n \mid a_{n+1} \neq a_{n+2}\} \setminus \{1,\dots, N\}\big) < \epsilon.
\end{equation}
For this $N$, by Lemma~\ref{L:rul}(i), 
there exits $\gamma > 0$ such that if $0\leq r < \gamma$, then the first $N+2$ digits of $r$ are all zeros. 
It follows that if $r \ominus s, \; s \ominus t < \gamma$, then the first $N+2$ digits of $r \ominus s$ and $s \ominus t$ are all zeros, therefore, $r \ominus t$ has the first $N+1$ digits equal to zero. 
This implies that $j(r \ominus t)$ contains no elements from $\{1,\dots,N\}$.
Furthermore, from the definition of $\ominus$, we have 
\begin{equation}\label{E:trio}
r \ominus s < \epsilon, \; s\ominus t < \epsilon \Rightarrow r \ominus t < 2\epsilon. 
\end{equation}
Finally, let 
\[
\delta' = \min\{\gamma, \delta, \epsilon\}.
\]
Hence, by our choice of $\delta'$, \eqref{E:minus1}, \eqref{E:tail}, and \eqref{E:trio}, we get 
\[
\rho(r,s), \; \rho(s,t) < \delta' \Rightarrow \rho(r,t) < 6\epsilon,
\]
and the lemma follows. 
\end{proof}

\begin{lemma}\label{L:tar} 
There exists a metric $d$ on $H$ such that 
 for each $\epsilon >0$, there exists $\delta>0$ with
\[
\big(\rho(r,s)<\delta \Rightarrow d(r,s)<\epsilon\big) \hbox{ and } \big(d(r,s)<\delta \Rightarrow \rho(r,s)<\epsilon\big), \hbox{  for all }r,s\in H.
\]
\end{lemma} 

\begin{proof} It suffices to check that $\rho$ fulfills properties (i)--(iii) from the assumption of Lemma~\ref{L:metr}. 
It is clear that $\rho$ has (i) and (ii); $\rho$ satisfies (iii) by Lemma~\ref{L:rhotri}.
\end{proof} 

Consider the topology induced on $H$ by the metric $d$ from Lemma~\ref{L:tar}. We call this topology the {\bf submeasure topology}. 
The definitions of $\rho$, $d$, and, therefore, also the submeasure topology depend on $\phi$ with 
$I= {\rm Exh}(\phi)$. Lemma~\ref{L:phs}(i) implies that these objects depend only on $I$ and not on the choice of $\phi$.

\subsubsection{The Polish module}

Recall that we have a base sequence $\vv a$ and an ideal $I$ such that $\vv a$ is adapted to $I$.  
We assume now that $I$ is a translation invariant analytic P-ideal $I$. 
We are ready to define the ${\mathbb Q}_{\vv a}$-module 
\[
I[{\vv a}].
\]
The underlying set of $I[{\vv a}]$ is $H$ from \eqref{E:hh}.
By Lemma~\ref{L:base1}, $H$ with the operation of $+$ and multiplication by elements of ${\mathbb Q}_{\vv a}$ inherited from $\mathbb R$ is a 
${\mathbb Q}_{\vv a}$-module. We topologize $H$ with the submeasure topology.

\begin{theorem}\label{T:C} 
Let $I$ be a translation invariant analytic P-ideal of subsets of $\mathbb N$, and let a base sequence $\vv a$ be adapted to $I$. Then $I[{\vv a}]$ is a Polish 
${\mathbb Q}_{\vv a}$-module and the identity map 
$I[{\vv a}]\to {\mathbb R}$ is a continuous ${\mathbb Q}_{\vv a}$-embedding. 
\end{theorem} 

We isolate the following result as a lemma as we will use it twice. 

\begin{lemma}\label{L:separ} 
Let $F$ be the collection of elements of ${\mathbb R}_{\geq 0}$ that have finitely many non-zero digits. If $I$ has an element that is an infinite subset of $\mathbb N$, 
then $F \cup -F$ is dense in $I[\vv a]$ with the submeasure topology. 
\end{lemma}

\begin{proof} 
It suffices to show that $F$ is dense in $H_{\geq 0}$.
For $r \in H_{\geq 0}$, recall $[r]_n$ from \eqref{E:trucr}. Clearly $[r]_n$ is in $F$ for every $n$. We verify that there is a subsequence of 
$([r]_n)_{n \in \mathbb N}$ that converges to $r$ in the submeasure topology.
If $r$ is in $F$, we have $[r]_n = r$ for all but finitely many $n$. If $r$ is not in $F$, there are two cases. 

When $j(r)$ is finite, fix a sequence $(n_i)$ of natural numbers such that 
$\phi(\{n_i\}) \to 0$. The existence of such a sequence follows from the assumption that an infinite subset of $\mathbb N$ is in $I$. Since $j(r-[r]_{n_i})= \{n_i\}$ 
for all but finitely many $i$, 
\begin{equation}\label{E:ni} 
\phi\big(j(r-[r]_{n_i})\big) \to 0\;\hbox{ as }i \to \infty. 
\end{equation}

Now assume $j(r)$ is infinite. Note that, for $n \in j(r)$, 
$j(r-[r]_n)$ is equal to one of the following two sets 
\[
j(r)\setminus \{1,\dots,n-1\} \text{, } j(r)\setminus \{1,\dots,n\}
\]
Thus, since $j(r)\in {\rm Exh}(\phi)$, we get 
\begin{equation}\label{E:ni2}
\phi\big(j(r-[r]_n)\big) \to 0\;\hbox{ as }j(r)\ni n\to\infty.
\end{equation} 

Furthermore, $(r-[r]_n) \to 0$ in $\mathbb R$ as $n\to \infty$, so, 
in both cases, by \eqref{E:ni} and \eqref{E:ni2}, we get a subsequence of $([r]_n)$ that converges to $r$ in terms of $\rho$, and this implies convergence in the submeasure topology 
by Lemma~\ref{L:tar}.
\end{proof}

\begin{proof}[Proof of Theorem~\ref{T:C}]
We need to see that the submeasure topology on $I[{\vv a}]$ is Polish. 
If $I$ is the ideal of all finite subsets of $\mathbb N$, then $I[\vv a]$ is countable and the submeasure topology is discrete, so $I[\vv a]$ is Polish. The first observation is immediate. 
To see the second one, fix a lsc submeasure $\phi$ such that $I = {\rm Exh}(\phi)$ and $\phi(\{n\}) > 1$ for all $n$. As pointed out, the submeasure topology does not depend on the 
choice of $\phi$. Let $\epsilon = \frac{1}{a_1}$. Then, for $r,s \in H$ with $\rho(r,s) < \epsilon$, we can conclude from the definition of $\rho$ and Lemma~\ref{L:rul}(i) that 
$[r \ominus s] = 0$, the first digit of $r\ominus s$ is $0$, and $j(r \ominus s) = \emptyset$.
These properties together imply that $r \ominus s = 0$. By Lemma~\ref{L:tar}, this implies the topology induced by $d$ is discrete. 

Now we assume $I$ is not the ideal of finite sets, or equivalently, $I$ has an infinite subset of $\mathbb N$ as element. Separability follows from 
Lemma~\ref{L:separ}. 
It remains to see that the metric $d$ is complete, that is, we need to see that each $d$-Cauchy sequence in $I[{\vv a}]$ 
converges with respect to $d$ to an element of $I[{\vv a}]$. It will suffice to do it for sequences contained in $H_{\geq 0}$ and, by Lemma~\ref{L:tar}, with Cauchy-ness and convergence 
understood in terms of $\rho$. So let $(r(n))$ be a sequence in $H_{\geq 0}$ such that, for each $\epsilon>0$, $\rho(r(m), r(n))< \epsilon$ for large $m,n$. 
By the definition of $\rho$, $(r(n))$ is Cauchy in $\mathbb R$, so it converges to some $r \in \mathbb R_{\geq 0}$. 
We need to show that $r \in H_{\geq 0}$ and $\rho(r(n), r)\to 0$. To do that, we first prove the following claim.
\begin{claim*}\label{Cl:cplt} 
Given $\epsilon>0$, for each large enough $M\in {\mathbb N}$, we have 
\[
\phi\big(j(r(n))\setminus \{1,\dots,M\}\big) < \epsilon\;\hbox{ for large }n.
\]
\end{claim*}
\noindent{\em Proof of Claim.}  
Using Lemma~\ref{L:rhotri}, for the given $\epsilon$, we fix $\delta > 0$ such that for all $s,t,u \in H_{\geq 0} $
\begin{equation} \label{E:semitri}
\rho(s,t), \; \rho(t,u) < \delta \Rightarrow \rho(s,u) < \epsilon/2.
\end{equation}
By Cauchy-ness of $(r(n))$, there exists $k$ such that  
\[
\rho(r(n),r(k)) < \delta \hbox{ for all  } n \geq k. 
\]
Since $r(k)$ is in $H_{\geq 0}$ and, from Lemma~\ref{L:separ}, $F$ is dense in $H_{\geq 0}$, there is a number $s$ with finitely many non-zero digits such that 
\[
\rho(r(k), s))  < \delta.
\] 
By \eqref{E:semitri} and the two inequalities above, $\rho(r(n), s)< \epsilon/2$ for all $n \geq k$. In particular, 
\begin{equation}\label{E:rngeqk}
\phi\big(j(r(n) \ominus s)\big) < \frac{\epsilon}{2}, \hbox{ for all }n \geq k.
\end{equation}

Notice that, for each $n$, either $r(n) \ominus s = r(n) - s$ or $r(n) \ominus s = s - r(n)$. 
In the first case, it is easy to see that subtracting $s$ only affects finitely many digits of $r(n)$, so, for large enough $M$, which depends only on $s$,
\begin{equation}\label{E:rns1}
j\big(r(n)\big)\setminus \{1,\dots,M\} \subseteq j\big(r(n) - s\big).
\end{equation}  
In the second case, it is not hard to see that for large enough $M$, which also only depends on $s$,
\begin{equation}\label{E:rns2}
j\big(r(n)\big)\setminus \{1,\dots,M\} \subseteq j\big(s-r(n)\big) \cup \{n \mid a_n \neq a_{n+1}\}.
\end{equation} 
Further, since $\{n \mid a_n \neq a_{n+1}\} \in I$, there is $N \in \mathbb N$ such that
\[
\phi\big(\{n \mid a_n \neq a_{n+1}\} \setminus \{1,\dots,N\}\big) < \epsilon/2. 
\]
Hence, putting together both cases, we see that for $M$ such that \eqref{E:rns1} and \eqref{E:rns2} hold and, in the second case, $M\geq N$, we have 
\[
\phi\big(j(r(n))\setminus \{1,\dots,M\}\big)  < \phi\big(j(r(n)\ominus s)\big)+\epsilon/2.
\]
This inequality implies the conclusion of the claim by \eqref{E:rngeqk} .

\smallskip
For contradiction, suppose that $r \notin H_{\geq 0}$, that is, $j(r) \notin I$ since $r$ is non-negative. Then 
there exists $\epsilon > 0$ such that, for each $m$, there is a finite set $P_m \subseteq j(r)$ with $\min(P_m) > m$ and $\phi(P_m) \geq \epsilon$. 
By Lemma~\ref{L:contj}(i) and the fact that $r$ has infinitely many non-zero digits (as otherwise $r\in H_{\geq 0}$), 
each $P_m$ is included in $j(r(n))$ for large enough $n$, which contradicts Claim when $m = M$. Hence $r$ is in $H_{\geq 0}$.

We now show that $\rho(r(n), r)\to 0$. Since $r(n)\ominus r\to 0$ in $\mathbb R$ as $n\to \infty$, it is enough to see that 
\begin{equation}\label{E:cauc}
\phi\big( j(r(n)\ominus r)\big) \to 0 \hbox{ as }n\to\infty. 
\end{equation} 
Assume towards a contradiction that this is not the case, so we can fix $\epsilon>0$ with 
\begin{equation}\label{E:bel}
\phi\big( j(r(n)\ominus r)\big) >\epsilon\hbox{ for infinitely many }n.
\end{equation}
By Lemma~\ref{L:rul}(i), we have that for each $i$, for large $n$, $(r(n)\ominus r)_i = 0$.
Thus, we get that, for each $M$, 
\begin{equation}\label{E:bel2}
\{ 1, \dots, M\} \cap  j\big(r(n)\ominus r\big) = \emptyset\;\hbox{ for large }n. 
\end{equation}
By Lemma~\ref{L:jalg}, for each $n$, $j(r(n)\ominus r)$ is included in the union of the following sets 
\[
\begin{split} 
&j(r(n)),\; j(r),\;  \{k \mid a_k \neq a_{k+1}\} ,\\  
&\{k \mid k+1 \in j\big(r(n)\big)\},\; \{ k\mid k+1\in j(r)\},\; \{k \mid a_{k+1} \neq a_{k+2}\}.
\end{split}
 \]
Thus, by \eqref{E:bel}, \eqref{E:bel2}, and Lemma~\ref{L:phs}(iii), there exists $\delta>0$ such that, for each $M$, one of the following statements is true 
\begin{equation}\label{E:three}
\begin{split} 
\phi\big( j(r(n)\big)\setminus \{ 1, \dots, M\} \big) &>\delta\;\hbox{ for infinitely many }n,\\
\phi\big( j(r)\setminus \{ 1, \dots, M\} \big) &>\delta,\\
\phi\big( \{ k\mid a_k\neq a_{k+1}\}\setminus \{ 1, \dots, M\}  \big) &>\delta.
\end{split} 
\end{equation} 
Since $j(r), \{ k\mid a_k\neq a_{k+1}\}\in I$, for large $M$, the last two inequalities of \eqref{E:three} fail. Hence, the first inequality holds for large $M$ 
contradicting Claim. So \eqref{E:cauc} is proved.

Note that, by Lemma~\ref{L:tar} and the definition of $\rho$, the submeasure topology on $I[{\vv a}]$ contains the topology it inherits from $\mathbb R$. 
Thus, the inclusion map $I[{\vv a}]\to {\mathbb R}$ is continuous. Since the algebraic operations on $H$ are inherited from $\mathbb R$, the inclusion map 
is a  ${\mathbb Q}_{\vv a}$-embedding.

We prove that $I[{\vv a}]$ with $+$ is a topological group with the submeasure topology. Since the submeasure topology is Polish, 
by \cite[Corollary~9.15]{Ke}, it suffices to show that the functions 
\[
I[{\vv a}]\ni s\to r+s\in I[{\vv a}], \hbox{ for each }r\in I[{\vv a}]
\]
and 
\[
I[{\vv a}]\ni s\to -s\in I[{\vv a}]
\]
are continuous. Observe that $\rho$ has the following invariance properties 
\[
\rho(s_1, s_2) = \rho(r+ s_1, r+s_2)\;\hbox{ and }\; \rho(s_1, s_2) = \rho(-s_1, -s_2)
\]
for all $r, s_1, s_2\in I[{\vv a}]$. The continuity of the functions above is now an immediate consequence of the invariance properties of $\rho$ 
and Lemma~\ref{L:tar}. 

Finally, we show that multiplication by elements of ${\mathbb Q}_{\vv a}$ is continuous on $I[{\vv a}]$ taken with the submeasure topology. 
Observe first that since the submeasure topology is Polish and the identity map 
$I[{\vv a}]\to {\mathbb R}$ is continuous, by \cite[Corollary 15.2]{Ke}, the submeasure topology has the same Borel sets as the topology $I[{\vv a}]$ inherited from $\mathbb R$.
Now, multiplication by an element of ${\mathbb Q}_{\vv a}$, being a Borel map on $I[{\vv a}]$ 
taken with the topology inherited from $\mathbb R$, is also a Borel map on $I[{\vv a}]$ taken with the submeasure topology. 
It is clearly also a group homomorphism. 
Thus, since the submeasure topology is Polish, by Pettis Theorem \cite[Theorem 9.10]{Ke}, this map is continuous. 
\end{proof}

\section{Inclusions and homomorphisms among modules}\label{S:3}

We will need two more properties of ideals and base sequences.  
An ideal of subsets of $\mathbb N$ is {\bf tall} if each infinite subset of $\mathbb N$ contains an infinite subset in $I$. 
If $I$ is of the form $I= {\rm Exh}(\phi)$ for a lsc submeasure $\phi$, then it is easy to check that $I$ is tall precisely when 
\begin{equation}\label{E:tal} 
\phi\big( \{ n\}\big)\to 0\;\hbox{ as }n\to \infty. 
\end{equation}
We call a base sequence $\vv{a}$ {\bf uniform} provided that, for each prime $p$, if $p\mid a_n$ for some $n$, then $p\mid a_n$ for all but finitely many $n$; in other words, 
\[
{\rm pr}({\vv{a}}) = \{ p \mid \; p \hbox{ a prime and } p\mid a_n \hbox{ for some }n\}.
\]

\subsection{Inclusions and non-inclusions}

In this section, we view modules of the form $I[{\vv a}]$ simply as subsets of $\mathbb R$. We are interested in the inclusion relation 
among such sets as the ideal $I$ varies.

First, we have an observation that inclusion between ideals implies inclusion between the corresponding modules. 

\begin{proposition}\label{P:incl} 
Let $I,J$ be translation invariant analytic P-ideals, and let $\vv a$ be a base sequence adapted to both $I$ and $J$. 
If $I \subseteq J$, then $I[{\vv a}]\subseteq J[{\vv a}]$. 
\end{proposition} 

\begin{proof} The inclusion $I[{\vv a}]\subseteq J[{\vv a}]$ follows directly from the assumption $I\subseteq J$. 
\end{proof}

Now, we prove a theorem asserting that a slightly stronger notion of non-inclusion between ideals implies non-inclusion between the corresponding modules, even if we allow the modules 
to be multiplied by non-zero real numbers. We need the following two definitions.
Let $I$ and $J$ be analytic P-ideals. We say that $I$ {\bf is not included in $J$ on intervals} if there are lsc submeasures $\phi$ and $\psi$ with $I= {\rm Exh}(\phi)$ and $J = {\rm Exh}(\psi)$, for which 
there exists $d>0$ such that 
\[
\inf\{ \phi(P)\mid P \hbox{ an interval in }{\mathbb N} \hbox{ with }\psi(P)\geq d\} =0.
\]
Notice that this property implies $I \nsubseteq J$.

For a subset $X$ of $\mathbb R$ and a real number $c$, we write 
\[
c\, X = \{ cx\mid x\in X\}. 
\]

\setcounter{claim}{0}

\begin{theorem} \label{T:D}
Let $I,J$ be translation invariant, analytic P-ideals with $I$ being tall. 
Let $\vv a$ be a base sequence adapted to both $I$ and $J$. 
If $I$ is not included in $J$ on intervals, 
then $c \, I[{\vv a}]\not\subseteq J[{\vv a}]$, for each non-zero real number $c$. 
\end{theorem}

\begin{proof} 
Since $I[{\vv a}]$ is closed under addition and taking additive inverses, it will suffice to show the conclusion only for $c\geq 1$. 
So we assume $c\geq 1$, and set 
\[
C= c \, I[{\vv a}], 
\]
with the aim to show that $C\not\subseteq J[{\vv a}]$. 

Since $I$ is not included in $J$ on intervals, there are lsc submeasures $\phi$ and $\psi$ such that $I={\rm Exh(\phi)}$ and $J={\rm Exh(\psi)}$, for which there 
exist $d > 0$ and a sequence of intervals $P_n=[l_n,m_n]$, with $m_n < l_{n+1}$, such that 
\begin{equation}\label{E:smla}
\psi(P_n) > d\;\hbox{ and }\;\phi(P_n) < 2^{-n}.
\end{equation} 
We also set $j= j_{\vv a}$.

\begin{claim}\label{Cl:y} 
Given $n \in \mathbb N$, there is $y \in C$ such that 
\begin{enumerate} 
\item[(i)]  $j(y) \supseteq P_n$; 

\item[(ii)]  $0\leq y < \frac{1}{a_1a_2\cdots a_{l_n}}$; 

\item[(iii)] $y_{m_n+2} \neq 0$.
\end{enumerate} 
\end{claim} 

\noindent{\em Proof of Claim~\ref{Cl:y}.} To find such $y$, pick $z\in [0,1)$ so that 
 \begin{equation}\label{E:jz}
         j(z) \cap [1,m_n]  = P_n;\;z_{m_n+2} \neq 0; \; z_i = 0,\hbox{ for all }i \in {\mathbb N} \setminus [l_n+1, m_n+2]. 
 \end{equation}
Note that from \eqref{E:jz}, we have 
\begin{equation}\label{E:zsm} 
        0< z < \frac{1}{a_1 a_2 \cdots a_{l_n}}.
\end{equation} 
Since $C$ is dense in $\mathbb R$, 
we can approximate $z$ from above by elements $y$ of $C$. For a tight enough such approximation $y$, we get (ii) from \eqref{E:zsm}, (i) 
from \eqref{E:jz} and Lemma~\ref{L:contj}(ii), and (iii) from \eqref{E:jz} and Lemma~\ref{L:conv digi}(ii). The claim follows.

\begin{claim}\label{Cl:ant} 
Given $\epsilon>0$, for large enough $n\in {\mathbb N}$ the following statement holds. For each $v\in I({\vv a})$ with $0\leq v < \frac{1}{a_1a_2\cdots a_{l_n}}$, there exists 
$w\in  I({\vv a})$ such that 
\begin{enumerate} 
\item[(i)] $0\leq w\leq v< w+\frac{1}{c\, a_1a_2\cdots a_{m_n+2}}$; 

\item[(ii)] $\phi\big( j(w)\big) <\epsilon$. 
\end{enumerate} 
\end{claim}

\noindent{\em Proof of Claim~\ref{Cl:ant}.} 
Let $k_0\in {\mathbb N}$ be such that $2^{k_0}>c$. For large enough $n$, we have 
\begin{equation}\label{E:tne}
2^{-n}<\frac{\epsilon}{3}
\end{equation} 
and, by tallness of $I$ applied using condition \eqref{E:tal}, 
\begin{equation}\label{E:x}
\phi(X)<\frac{\epsilon}{3},\hbox{ for each set }X\hbox{ of size }\leq k_0+3 \hbox{ with } i> m_n\hbox{ for all }i\in X.
\end{equation}

Fix $n$ with \eqref{E:tne} and \eqref{E:x}. Since $v\in I[{\vv a}]$ and $v \geq 0$, we have $j(v) \in I$, so there is $N \in \mathbb N$ with $N > m_n+k_0+2$ and such that
\begin{equation}\label{E:sm2} 
\phi\big(j(v)\setminus [1,N] \big) < \epsilon/3.
\end{equation}
Let 
\[
K= [ m_n+1, \; m_n+k_0+2],
\]
and 
define $w$ to be the real number such that 
\begin{equation}\notag
w_i= 
\begin{cases}
v_i, &\hbox{ if } i \leq m_n+k_0+2 \hbox{ or } i \geq N+1,\\
0,  &\hbox{ otherwise}.
\end{cases}
\end{equation}
Since $v$ and $w$ differ on finitely many digits and $v\in  I[{\vv a}]$, it follows from the definition of the underlying set of $I[{\vv a}]$ that $w\in  I[{\vv a}]$. 

We check that $w$ satisfies the two properties in the claim. First, directly from the definition we have 
\begin{equation}\notag
0 \leq v-w < \frac{1}{a_1 \cdots a_{m_n+2} \cdots a_{m_n+k_0+2}}\leq \frac{1}{a_1 \cdots a_{m_n+2} \, 2^{k_0}}.
\end{equation} 
from which, by our choice of $k_0$, we get (i).
Second, since $0\leq w \leq v < \frac{1}{a_1a_2\cdots a_{l_n}}$, using Lemma~\ref{L:rul}(i),  we see that the first $l_n$ digits of $w$ are 0. It follows that 
\begin{equation}\label{E:sm1}
j(w) \cap [1,l_n-1] = \emptyset. 
\end{equation}
Note that $j(w) \cap [m_n+1,N]$ is a subset of $K  \cup \{N\}$. Hence, by \eqref{E:x}, we get 
\begin{equation}\label{E:sm3}
\phi\big(j(w) \cap [m_n+1,N]\big) \leq \phi\big(K \cup \{N\}\big) < \epsilon/3.
\end{equation} 
From the definition of $w$ and by \eqref{E:sm2}, we have 
\begin{equation}\label{E:sm4}
\phi\big(j(w) \setminus [1,N]\big) = \phi\big(j(v)\setminus [1,N]\big) < \epsilon/3.   
\end{equation}
Putting together \eqref{E:sm1}, \eqref{E:sm3} and \eqref{E:sm4}, we get 
\[
\phi(j(w)) < 0 + \phi(P_n) + \epsilon/3+\epsilon/3 = \phi(P_n) + 2\epsilon/3.
\]
The above inequality implies the conclusion of the claim by  \eqref{E:smla}  and \eqref{E:tne}.

\smallskip
We use Claims~\ref{Cl:y} and \ref{Cl:ant} to show that given $\epsilon > 0$, there exists $w \in I[{\vv a}]$ such that 
 \begin{enumerate}
            \item[(a)] $0\leq w< \epsilon$;
            \item[(b)] $\phi\big(j(w)\big) < \epsilon$;
            \item[(c)] $\psi\big(j(c\,w)\big) > d$.
\end{enumerate}

Let $\epsilon>0$. Fix $n$ large enough so the conclusion of Claim~\ref{Cl:ant} holds and 
\begin{equation}\label{E:eam} 
\frac{1}{a_1a_2\cdots a_{l_n}}<\epsilon.
\end{equation} 
For this $n$, Claim~\ref{Cl:y} produces $y\in C$ with properties (i)--(iii) from that claim. 
Set 
\[
v=   \frac{y}{c}. 
\]
Note that $v \in I[{\vv a}]$ since $y\in C$. Using the inequality $c\geq 1$ and Claim~\ref{Cl:y}(ii), we get
\begin{equation}\label{E:ine}
0\leq v \leq y< \frac{1}{a_1 a_2 \cdots a_{l_n}}.
\end{equation} 
So Claim~\ref{Cl:ant} applies to $v$ producing $w\in I[{\vv a}]$ with properties (i) and (ii) from that claim.

Since, by Claim~\ref{Cl:ant}(i), $0\leq w\leq v$, we get (a) from \eqref{E:eam} and \eqref{E:ine}. Point (b) is immediate 
from Claim~\ref{Cl:ant}(ii). By Claim~\ref{Cl:ant}(i), we have 
\[
0\leq c\, w\leq y< c\, w+\frac{1}{a_1a_2\cdots a_{m_n+2}}.
\]
These inequalities imply (c) by Claim~\ref{Cl:y}(i) and (iii), Lemma~\ref{L:rul}(ii), and \eqref{E:smla}.

Now, to finish the proof of the theorem, assume for contradiction that $C \subseteq J[{\vv a}]$. Then $x \rightarrow cx$ defines a function from $I[{\vv a}]$ to $J[{\vv a}]$. 
The map $x \rightarrow cx$, being a Borel map from $\mathbb R$ to $\mathbb R$, is also a Borel map from $I[{\vv a}]$ to $J[{\vv a}]$ both taken with the submeasure topologies by 
\cite[Corollary~15.2]{Ke}. 
Moreover, it is a group homomorphism, so by Pettis Theorem \cite[Theorem~9.10]{Ke}, it is continuous. However, the existence of $w\in I[{\vv a}]$ for each $\epsilon>0$  with 
properties (a)--(c) implies, by Lemma~\ref{L:tar}, that this map is discontinuous at $0$, yielding a contradiction. 
\end{proof}

\subsection{Consequences concerning module homomorphisms}\label{S:4} 

We derive some consequences of Theorem~\ref{T:D} for homomorphisms among modules constructed in Section~\ref{S:2}. 
We will use these consequences to prove Theorem~\ref{T:B}.

\begin{lemma}\label{C:noni}
Let $I,J$ be translation invariant, analytic P-ideals, with $I$ being tall.  Let $\vv a$ be a uniform base sequence adapted to both $I$ and $J$. 
If $I$ is not included in $J$ on intervals, then each continuous ${\mathbb Q}_{\vv a}$-module homomorphism from $I[{\vv a}]$ to $J[{\vv a}]$ is identically equal to $0$. 
\end{lemma}

To prove Lemma~\ref{C:noni}, we need a result on the form of continuous module homomorphisms.

\begin{lemma}\label{L:hom} 
Let $I, J$ be translation invariant, analytic P-ideals, with $I$ having an infinite subset of $\mathbb N$ as element. Let $\vv a$ be a uniform base sequence adapted to $I$ and $J$. 
If $f\colon I[{\vv a}]\to J[{\vv a}]$ is a continuous ${\mathbb Q}_{\vv a}$-module homomorphism, then there exists $c\in {\mathbb R}$ with 
$f(y)=c\, y$ for all $y\in I[{\vv a}]$. 
\end{lemma}

\begin{proof} 
We begin with showing that ${\mathbb Q}_{\vv a}$ is dense in $I[{\vv a}]$. 
Let $F$ be the collection of elements of ${\mathbb R}_{\geq 0}$ that have finitely many non-zero digits. 
By Lemma~\ref{L:separ}, $F \cup -F$ is dense in $I[\vv{a}]$. We show $F \cup -F = {\mathbb Q}_{\vv a}$. 

Suppose $r$ is non-negative and in ${\mathbb Q}_{\vv a}$, that is, $r= \frac{k}{l}$, where $k\in {\mathbb N}$ and $l$ is a product of primes in ${\rm pr}({\vv{a}})$. 
Each prime in ${\rm pr}({\vv{a}})$ divides $a_n$ for all but finitely many $n$. Thus, there is $N$ such that $l \mid a_1a_2\cdots a_N$. Therefore, $a_1a_2\cdots a_N r$ is an integer, 
which implies that $r_n = 0$ for all $n > N$. It follows that $r$ is in $F$. By the same argument, we see that if 
$r \in {\mathbb Q}_{\vv a} $ is non-positive, then $r$ is in $-F$. 

Conversely, if $r$ is in $F \cup -F$, there exists $N$ such that $a_1a_2\cdots a_Nr$ is an integer. Hence $r$ is of the form $\frac{k}{a_1a_2\cdots a_N}$. Each prime $p$ that 
divides $a_1 a_2\cdots a_N$ divides some $a_n$, so $p$ is in ${\rm pr}({\vv{a}})$ by the uniformity of $\vv a$. Therefore, $r$ is in  
${\mathbb Q}_{\vv a} $. 

We continue our proof of Lemma~\ref{L:hom}. 
Note that $1$ is in $I[{\vv a}]$ since $1$ has finitely many non-zero digits. Set 
\[
c= f(1). 
\]
Note that for each $y\in  {\mathbb Q}_{\vv a}$, we have 
\begin{equation}\label{E:cru} 
f(y) = f(y1)= yf(1) =  c\, y,
\end{equation}
because $f$ is a ${\mathbb Q}_{\vv a}$-module homomorphism. 

Since $f$ is continuous as a function from $I[{\vv a}]$ to $J[{\vv a}]$, and the topology on $J[{\vv a}]$ contains the topology it inherits from $\mathbb R$, 
$f$ is continuous as a function from $I[{\vv a}]$ to $\mathbb R$. Similarly, since the topology on $I[{\vv a}]$ contains the topology it inherits from $\mathbb R$, 
the function $y\to c\,y$ is continuous as a function from $I[{\vv a}]$ to ${\mathbb R}$. It follows that \eqref{E:cru} holds for all $y$ in the closure of ${\mathbb Q}_{\vv a}$ 
in $I[{\vv a}]$, which gives the lemma as ${\mathbb Q}_{\vv a}$ is dense in $I[{\vv a}]$. 
\end{proof}

\begin{proof}[Proof of Lemma~\ref{C:noni}]
The conclusion of the corollary follows from Theorem~\ref{T:D} and Lemma~\ref{L:hom}. 
\end{proof}

\section{Proof of Theorem~\ref{T:B}}  

We construct a family $I_x$, $x\subseteq {\mathbb N}$, of analytic P-ideals that will allow us to apply Lemma~\ref{C:noni} to obtain Theorems~\ref{T:B}. 
This family is probably the simplest one that can serve our purpose, but some more complicated families from the literature would also work; see, for example, \cite{LV}.

For $k\in {\mathbb N}$, let 
\[
P_k = \{ n\in {\mathbb N}\mid 2^{k-1}\leq n<2^k\}. 
\]
We note the following properties of the intervals $P_k$: 
\begin{enumerate}
\item[(a)] $\sum_{n\in P_k} \frac{1}{n} \geq \frac{1}{2}$; 

\item[(b)] $\sum_{n\in a} \frac{1}{n} <\infty$, for each $a\subseteq {\mathbb N}$ that has at most one point in common with $P_k$ for each $k$. 
\end{enumerate}
For $x\subseteq {\mathbb N}$, set $A_x= \bigcup_{k\in x}P_k$, and let, for $a\subseteq {\mathbb N}$, 
\[
\phi_x(a) = \sum_{n\in a\cap A_x} \frac{1}{2^n}  + \sum_{n\in a\setminus A_x} \frac{1}{n}. 
\]
It is clear that $\phi_x$ is a lsc submeasure. 
Finally, define 
\[
I_x = {\rm Exh}(\phi_x) 
\]
and observe that 
\[
I_x= \{ a\subseteq {\mathbb N}\mid \phi_x(a)<\infty\}. 
\]
It is then clear that $I_x$ is an analytic P-ideal and, using (b), that it is translation invariant and tall. 
Furthermore, observe that, for $x, y\subseteq {\mathbb N}$, 
\begin{equation}\label{E:ime1}
x\setminus y \hbox{ finite} \Longrightarrow I_x\subseteq I_y,
\end{equation} 
\begin{equation}\label{E:ime2}
x\setminus y \hbox{ infinite} \Longrightarrow I_x \hbox{ is not included in } I_y \hbox{ on intervals}.
\end{equation} 
The first implication is clear since the assumption gives that $A_x\setminus A_y$ is finite. To see the second implication, note that if 
$x\setminus y$ is infinite, then the sequence of intervals $(P_k)_{k\in x\setminus y}$ witnesses the non-inclusion on intervals as, by (a), $\phi_y(P_k)\geq 1/2$ for $k\in x\setminus y$, 
while $\phi_x(P_k)\to 0$ for $x\setminus y\ni k\to \infty$.

\begin{proof}[Proof of Theorem~\ref{T:B}]
Let $R$ be a subring of $\mathbb Q$ not equal to $\mathbb Z$; that is, 
\[
R = P^{-1} {\mathbb Z}
\]
for a nonempty set $P$ of primes. It is easy to find a uniform base sequence $\vv {\bf a}$ with 
\[
P= {\rm pr}({\vv a})\;\hbox{ and  }\;\{ k\mid a_k\not= a_{k+1}\} \subseteq \{ 2^k\mid k\in {\mathbb N}\}. 
\]
Note that the first equality gives $R= {\mathbb Q}_{\vv a}$ and the second one ensures that $\vv a$ is adapted to $I_x$ for each $x\subseteq {\mathbb N}$. 

Then, for each $x\subseteq {\mathbb N}$, $I_x[{\vv a}]$ is a Polish $R$-module by Theorem~\ref{T:C}. 
Let $f_x\colon I_x[{\vv a}]\to {\mathbb R}$ be the identity map. This is a continuous $R$-embeding into $\mathbb R$ 
Note that since each $I_x$ is $F_\sigma$, it follows from Lemma~\ref{L:hcomp} that $I_x[{\vv a}]$ is an $F_\sigma$ as a subset of $\mathbb R$. 

Observe that, by Proposition~\ref{P:incl} and \eqref{E:ime1}, if $x\setminus y$ is finite, $x,y\subseteq {\mathbb N}$, then $I_x[{\vv a}]\subseteq I_y[{\vv a}]$ as subsets of $\mathbb R$. 
The identity map $I_x[{\vv a}]\to I_y[{\vv a}]$ is obviously an $R$-module homomorphism and it is continuous by Pettis Theorem \cite[Theorem 9.10]{Ke}. 
If $x\setminus y$ is infinite, then the conclusion follows from \eqref{E:ime2} and Lemma~\ref{C:noni}.
\end{proof}


\begin{thebibliography}{10}
\bibitem{AM} M.\;F.\;Atiyah, I.\;G.\;Macdonald, {\em Introduction to Commutative Algebra}, Addison-Wesley Publishing Co., 1969.

\bibitem{Ch} E.\;W.\;Chittenden, {\em On the equivalence of {\'e}cart and voisinage}, Trans. Amer. Math. Soc. 18 (1917), 161--166.

\bibitem{En} R.\;Engelking, {\em General Topology}, Second edition. Sigma Series in Pure Mathematics, 6, Heldermann Verlag, 1989.

\bibitem{FS} J.\;Frisch, F.\;Shinko, {\em A dichotomy for Polish modules}, Israel J.\;Math., to appear; [https://arxiv.org/abs/2009.05855].

\bibitem{Ke} A.\;S.\;Kechris, {\em Classical Descriptive Set Theory}, Graduate Texts in Mathematics, 156, Springer-Verlag, 1995.

\bibitem{LV} A.\;Louveau, B.\;Veli{\v c}kovi{\' c}, {\em A note on Borel equivalence relations}, Proc.\;Amer.\;Math.\;Soc. 120 (1994), 255--259.


\bibitem{So1} S.\;Solecki, {\em Analytic ideals and their applications}, Ann.\;Pure Appl.\;Logic 99 (1999), 51--72.

\bibitem{So2} S.\;Solecki, {\em Local inverses of Borel homomorphisms and analytic P-ideals}, Abstr.\;Appl.\;Anal. 3 (2005), 207--220. 

\end{thebibliography}
\end{document}